\documentclass{amsart}
\usepackage{amssymb,amsmath,amsfonts,amsthm,amsxtra,afterpage}
\usepackage[dvips,all]{xy}
\usepackage{xcolor}
\usepackage{mdframed}
\usepackage[normalem]{ulem}
\usepackage{enumerate}
\usepackage{enumerate}

\usepackage{todonotes}

\newtheorem{thm}{Theorem}[section]
\newtheorem{prop}[thm]{Proposition}
\newtheorem{lem}[thm]{Lemma}

\theoremstyle{definition}
\newtheorem{dfn}[thm]{Definition}

\theoremstyle{remark}
\newtheorem{rmk}[thm]{Remark}

\numberwithin{equation}{section}

\renewenvironment{proof}[1][\proofname]{\begin{trivlist}\item[\hskip \labelsep \itshape \bfseries #1{}\hspace{2ex}]}
{\qed\end{trivlist}}

\begin{document}
\title[Invariant Theory of Unipotent Group Actions.]{On the Invariant Theory of Unipotent Group Actions from a Geometric Perspective.}
\author{Stephen Maguire}
\email{maguire2@illinois.edu}

\begin{abstract}
This paper is a sequel to an earlier paper by the author.  In that paper, the author formulated the notion of a $ c(t) $-pair for the action of \linebreak $ \mathbb{G}_{a} \cong \operatorname{Spec}(k[t]) $ on a variety $ \operatorname{Spec}(A) $.  Namely a $ c(t) $-pair is a pair of elements $ g,h \in A $ such that $ g(t_{0} \ast x) = g(x)+c(t_{0})h(x) $ for $ (t_{0},x) \in \mathbb{G}_{a} \times \operatorname{Spec}(A) $.  A quasi-principle pair is a $ b(t) $-pair $ (g,h) $ such that $ \mathbf{ker}(b(t)) $ acts trivially on $ \operatorname{Spec}(A) $.  The existence of a quasi-principle pair is intimately related to whether there is an open sub-variety $ D(h) \subseteq \operatorname{Spec}(A) $ such that $ D(h) $ is a trivial bundle over $ D(h)//\mathbb{G}_{a} $.  We defined the pedestal ideal to be the ideal of $ A $ generated by $ 0 $ and all $ h $ such that there exists a $ g \in A $ such that $ (g,h) $ is a quasi-principle pair.  The large pedestal ideal is the ideal of $ A $ generated by $ h $ such that there is a non-zero, additive polynomial $ c(t) $ and a $ g \in A $ such that $ (g,h) $ is a $ c(t) $-pair.  We described some pathological situations that could happen for $ \mathbb{G}_{a} $-actions on varieties over fields $ k $ of positive characteristic that could not happen over fields of characteristic zero.  We also generalized van den Essen's algorithm to $ \mathbb{G}_{a} $-actions on varieties over fields of positive characteristic.

We generalize the notion of a $ c(t) $-pair to the notion of an $ \alpha $-pair for connected, unipotent, linear algebraic groups $ G $.  Here $ \alpha $ is a surjective endomorphism of $ G $.  In this paper we explore the theory of pairs for unipotent groups and how they relate to local triviality conditions, i.e., we show that the existence of quasi-principle $ \alpha $-pairs relates to whether a neighborhood $ D(H) \subseteq \operatorname{Spec}(A) $ is a trivial $ G $-bundle over $ G \backslash \backslash D(H) $.  We also show that the existence of an arbitrary $ \alpha $-pair relates to whether there is an fppf neighborhood $ U \to D(H) $ which is a trivial $ G $-bundle and which satisfies certain conditions related to $ \alpha $.  This relates the less intuitive notion of an $ \alpha $-pair to local triviality in the fppf topology.  We then give a generalized version of van den Essen's algorithm to vector groups and unipotent groups.  This generalized version works in any characteristic.  In the process we discuss some of the pathologies that occur for actions of unipotent groups.

We also explore what stability for connected, unipotent, linear algebraic groups should look like by exploring the concept of affine stability.  Namely, a point $ x $ is affine stable if $ x \in \operatorname{Spec}(A) \setminus \mathcal{V}(\mathfrak{P}(A)) $.  For any such point, there is a neighborhood $ D(H) $ containing $ x $, such that after a suitable modification, $ D(H) $ is a trivial $ G $-bundle over $ G \backslash \backslash D(H) $. 
\end{abstract}
\maketitle
\section{Introduction.}
Mathematicians who work with connected, unipotent, linear algebraic groups would like to find a desirable notion of stability.  If we were to re-explore the theory of stability for a one dimensional torus, then the following propositions emerge:
\begin{prop}
    Let $ \operatorname{Spec}(A) $ be an affine $ \mathbb{G}_{m} $-variety over an algebraically closed field $ k $ with action $ \beta $.  The following are equivalent:
    \begin{itemize}
        \item[a)] the variety $ D(h) $ is a trivial $ \mathbb{G}_{m} $ bundle,
        \item[b)] there is a $ \mathbb{G}_{m} $-equivariant, generically smooth, dominant, morphism \linebreak $ \Psi: D(h) \to \mathbb{G}_{m} $ with connected fibres,
        \item[c)] there is a rational function $ g/h^{e} \in A_{h}^{\ast} $, which is a semi-invariant of weight one, i.e, $ (g/h^{e})(w \bullet x) = w \left((g/h^{e})(x)\right) $ for all $ w \in \mathbb{G}_{m} $ and $ x \in D(h) $.
    \end{itemize}
\end{prop}
\begin{prop}
    Let $ \operatorname{Spec}(A) $ be a $ \mathbb{G}_{m} $-variety with action $ \beta $.  The following are equivalent:
    \begin{itemize}
        \item[a)] if $ q \in \mathbb{N} $, then there is a morphism $ \Phi:D(h) \to (\mathbb{G}_{m})^{q} $ such that the following diagram commutes:
            \begin{equation}
            \xymatrix{
                \mathbb{G}_{m} \times D(h) \ar[d]^{(\operatorname{id}_{\mathbb{G}_{m}}, \Phi)} \ar[rr]^{\beta} & & D(h) \ar[d]^{\Phi} \\
                \mathbb{G}_{m} \times \left(\mathbb{G}_{m}\right)^{q} \ar[rr]^{\gamma_{q}} & & \left(\mathbb{G}_{m}\right)^{q}
                }.
            \end{equation}
            Here $ (\mathbb{G}_{m})^{q} $ is the $ \mathbb{G}_{m} $-variety whose space is $ \mathbb{G}_{m} $, but with action $ \gamma_{q} $ which sends a point $ (w,s) $ to $ w^{q} s $.
        \item[b)] there is an $ e \in \mathbb{N}_{0} $ and $ g \in A_{h}^{\ast} $ such that $ g/h^{e} $ is a semi-invariant of weight $ q $, i.e., 
            \begin{equation*}
                (g/h^{e})(w \bullet x) = w^{q} (g/h^{e})(x),
            \end{equation*}
            for all $ w \in \mathbb{G}_{m} $ and $ x \in D(h) $.
    \end{itemize}
\end{prop}
We give the proofs of these Propositions in an appendix at the end of the paper.  If $ \beta $ is an action of $ \mathbb{G}_{m} $ on $ \operatorname{Spec}(A) $ of ray type and $ A_{1} $ generates $ A $, then $ \operatorname{Spec}(A)^{s} $ is covered by $ \mathbb{G}_{m} $-stable neighborhoods $ D(h) $ such that $ D(h) $ is a trivial $ \mathbb{G}_{m} $-bundle over $ D(h)//\mathbb{G}_{m} $.

If $ \beta $ is an action of $ \mathbb{G}_{a} $ on $ \operatorname{Spec}(A) $ and we ask when a $ \mathbb{G}_{a} $-stable neighborhood $ D(h) $ is a trivial $ \mathbb{G}_{a} $-bundle over $ D(h)//\mathbb{G}_{a} $, then we arrive at the notion of a principle pair.  The action is quasi-principle if there is an additive polynomial $ b(t) $ and a $ b(t) $-pair $ (g,h) $ such that $ \mathbf{ker}(b(t)) $ stabilizes $ \operatorname{Spec}(A) $.  For such an action we may replace $ \mathbb{G}_{a} $ by $ \mathbb{G}_{a}/\mathbf{ker}(b(t)) \cong \mathbb{G}_{a} $ and assume that the action is free.  Upon making this modification, the affine stable points are the set of points $ x $ such that there exists a $ \mathbb{G}_{a} $-stable neighborhood $ D(h) $ which is a trivial $ \mathbb{G}_{a} $-bundle over $ D(h)//\mathbb{G}_{a} $.

In the case of connected, unipotent, linear algebraic groups, a surjective endomorphism $ \alpha \in \operatorname{End}(G) $ plays the role of a surjective endomorphism of $ \mathbb{G}_{a} $ obtained from an additive polynomial $ c(t) $.  If $ s $ is the dimension of $ G $ as a variety, then it is possible to generalize the notion of a $ c(t) $-pair to an $ \alpha $-pair $ (\mathbf{g},\mathbf{h}) $, where $ \mathbf{g} $ and $ \mathbf{h} $ are $ s $-tuples of elements of $ A $.  If $ \alpha $ is the identity, then an $ \alpha $-pair is a principle pair.  An $ \alpha $-pair is quasi-principle if $ \ker(\alpha) $ acts trivially upon $ \operatorname{Spec}(A) $.  In this case we may similarly replace $ G $ by $ G/\ker(\alpha) $ and assume that the action of $ G $ is free.  Upon making this modification, the affine stable points are the set of points $ x $ such that there exists a $ G $-stable neighborhood $ D(H) $ which is a trivial $ G $-bundle over $ G \backslash \backslash D(H) $.  We explore the theory of affine stability and its geometric implications.  Why would we be interested in such a theory?  In order to construct quotients of a variety by connected, unipotent, linear algebraic groups we must know when an open sub-variety $ U $ of a $ G $-variety $ \operatorname{Spec}(A) $ over a field $ k $ has a quotient $ G \backslash \backslash U $ which is locally of finite type over $ k $.  Also, Hilbert's 14th problem does not have always have an affirmative answer for connected, unipotent, linear algebraic groups, and we would ideally like to have some procedure for answering whether it does.

Hilbert's 14th problem asks the following question.  Given a linear representation $ \beta: G \to \operatorname{GL}(\mathbf{V}) $ of a linear algebraic group $ G $ over a field $ k $, is the ring $ S_{k}(\mathbf{V}^{\ast})^{G} $ a finitely generated $ k $-algebra?  Nagata gave one of the earliest counterexamples to Hilbert's 14th problem involving vector groups.  This counterexample was a specific type of what has been dubbed ``the Nagata action''.  To construct the Nagata action, one starts with the following action $ \gamma $ of $ \mathbb{G}_{a}^{n} $ on $ \operatorname{Spec}(k[x_{1},\dots,x_{2n}]) $.  If $ k[t_{1},\dots,t_{n}] $ is the affine coordinate ring of $ \mathbb{G}_{a}^{n} $, then the co-action of $ \gamma $ is described below:
\begin{align*}
    x_{i} & \mapsto x_{i} \quad 1 \le i \le n \\
    x_{n+i} & \mapsto x_{n+i}+t_{i}x_{i} \quad 1 \le i \le n.
\end{align*}
Nagata then chose $ n $ points $ y_{i} $ of $ \mathbb{A}^{r}_{k} $ in general position.  If $ (a_{i,1},\dots,a_{i,r}) \in \mathbb{A}^{r}_{k} $ is the point $ y_{i} $, then the we obtain a linear algebraic sub-group isomorphic to $ \mathbb{G}_{a}^{n-r} $ from the affine hypersurface $ \mathcal{V}(\langle \sum_{i=1}^{n} a_{i,j} t_{i} \rangle_{j=1}^{r}) \subseteq \mathbb{G}_{a}^{n} $.  The induced action $ \beta $ of $ \mathbb{G}_{a}^{n-r} $ on $ \mathbb{A}^{2n}_{k} $ is the Nagata action for the parameters $ (n,r) $.  Nagata showed that if $ r $ is equal to three and $ n $ is greater than four, then $ k[X]^{\mathbb{G}_{a}^{n-3}} $ is not finitely generated.  Mukai expanded on Nagata's proof in \cite{MukaiCounterexample} by showing that the ring of invariants of the Nagata action is finitely generated if and only if
\begin{equation*}
    \frac{1}{r}+\frac{1}{n-r} \ge \frac{1}{2}.
\end{equation*}
Both proofs started by computing the ring of invariants on the open $ \mathbb{G}_{a}^{n-r} $-stable, subvariety $ D(x_{1}\cdots x_{n}) $ of $ \mathbf{V} $.  Nagata used this computation to write the ring of invariants as a direct sum of colon ideals.  Mukai then elaborated on this construction and proved that the ring of invariants is the Cox ring of a blow up of $ \mathbb{P}^{r-1}_{k} $ at $ n $ points in general position.

An open question is whether Hilbert's 14th problem has an affirmative answer if the linear algebraic group $ G $ is $ \mathbb{G}_{a}^{2} $.  Castravet and Tevelev \cite{CastravetTevelev} show that Nagata's original counterexample fails for the vector group $ \mathbb{G}_{a}^{2} $.

If $ \beta $ is an action of $ \mathbb{G}_{a} $ on a variety $ \operatorname{Spec}(A) $ over a field $ k $ of characteristic zero, and $ g,h \in A $ are elements such that $ g(t_{0} \ast x)=g(x)+t_{0} h(x) $ for a point $ (t_{0},x) \in \mathbb{G}_{a} \times \operatorname{Spec}(A) $, then van den Essen's algorithm allows one to explicitly compute the ring $ A_{h}^{\mathbb{G}_{a}} $.  One nice attribute of van den Essen's algorithm is that it allows one to work with the ring $ A^{\mathbb{G}_{a}}_{h} $ without knowing the specific action that one is working with.  So van den Essen's algorithm is a useful tool to work with actions of $ \mathbb{G}_{a} $ without having access to many details about a specific action.  Other Groebner basis algorithms like \cite{DerksenKemper} allow one to compute the ring of invariants via Groebner basis techniques.  However, one either needs to work with a specific underlying representation or have knowledge about what underlying properties a Groebner basis will have to work in more generality.

In the case of the Nagata action for the parameters $ (n,r) $, it is remarkably trivial to compute the ring $ k[X]_{x_{1}\cdots x_{n}}^{\mathbb{G}_{a}^{n-r}} $.  However, in most cases, such abstract computations become intractable.  Invariant theorists would like to have something similar to van den Essen's algorithm for connected, unipotent, linear algebraic groups.  If invariant theorists could compute the localization of the ring $ A^{G} $ at a suitable invariant $ H $, then they could perhaps write the underlying invariant ring as the Cox ring of a suitable variety $ Z $.  Then, they could examine the geometry of $ Z $ to attempt to determine whether $ A^{G} $ is a finitely generated $ k $-algebra.  However, to do this they would need to explicitly construct the localization of the ring of invariants.  Something akin to van den Essen's algorithm would be a great aid as a result.

In this paper, we generalize van den Essen's algorithm for all connected, unipotent, linear algebraic groups $ G $.  This algorithm works for all quasi-principle $ G $-varieties $ \operatorname{Spec}(A) $ over fields of arbitrary characteristic.

\section{A Generalization of the Theory of Pairs to Connected, Unipotent, Linear Algebraic Groups.}
\begin{dfn} \label{D:gpairGp}
    The monoid $ \mathbf{P}(G) $ is the sub-monoid of $ \operatorname{End}(G) $ generated by $ \alpha $ such that $ G/\ker(\alpha) \cong G $, i.e., the set of surjective endomorphisms of $ G $.  We call $ \mathbf{P}(G) $ the \emph{$ G $-pair monoid}.
\end{dfn}
Let $ G $ be a connected, unipotent, linear algebraic group acting on a variety $ \operatorname{Spec}(A) $ via a left action $ \beta $.  The group $ G $ is connected and unipotent.  As a result, if $ G^{(0)} = G $ and $ G^{(i)}=[G,G^{(i-1)}] $ for $ i>0 $, then $ G^{(i)}/G^{(i+1)} $ is a vector group.  So, if the dimension of $ G $ is $ s $, then there are $ s $ morphisms $ y_{i}: \mathbb{G}_{a} \to G $ such that the map which sends $ (a_{1},\dots,a_{s}) $ to $ y_{1}(a_{1}) \cdot y_{2}(a_{2}) \cdots y_{s}(a_{s}) $ is an isomorphism.  We shall make it a convention for the rest of this section to simply write $ (a_{1},\dots,a_{s}) $ for the image of a point $ (a_{1},\dots,a_{s}) $ in $ G $ under this isomorphism.

Throughout the next two sections, we will make the assumption that if $ G $ acts on $ \operatorname{Spec}(A) $ via a left action, and $ H $ is the maximal sub-group scheme of $ G $ such that $ H $ acts trivially upon $ \operatorname{Spec}(A) $, then $ H $ is finite.

\begin{dfn} \label{D:pairUnip}
    Let $ G $ be a connected, unipotent, linear algebraic group acting on a variety $ \operatorname{Spec}(A) $ via a left action $ \beta $.  If $ \alpha \in \mathbf{P}(G) $, then an $ \alpha $-pair is a collection $ ((g_{1},\dots,g_{s}),(h_{1},\dots,h_{s})) $ (or $ (\mathbf{g},\mathbf{h}) $ for simplicity) such that if $ H $ is equal to $ \prod_{i=1}^{s} h_{i} $ and $ (w,x) \in G \times D(H) $, then
    \begin{itemize}
        \item[i)] the following identity holds: 
            \begin{equation*}
                ((g_{1}/h_{1})(w \ast x),\dots,(g_{s}/h_{s})(w \ast x)) = \alpha(w) \cdot ((g_{1}/h_{1})(x),\dots,(g_{s}/h_{s})(x))
            \end{equation*}
        \item[ii)] the transcendence degree of $ k[g_{1}/h_{1},\dots,g_{s}/h_{s}] $ is equal to $ s $.
    \end{itemize}
    Here $ \cdot $ is the multiplication in $ G $.

    An $ \alpha $-pair is quasi-principle if $ \ker(\alpha) $ acts trivially upon $ \operatorname{Spec}(A) $.  An $ \alpha $-pair $ (\mathbf{g},\mathbf{h}) $ is principle if $ \alpha $ is the identity.

    The \emph{large pedestal ideal} of $ A $ is the ideal generated by $ \prod_{i=1}^{s} h_{i} $ such that there exists a $ \alpha \in \mathbf{P}(G) $ and $ s $-tuples $ (\mathbf{g},\mathbf{h}) $ equal to
    \begin{align*}
        \mathbf{g} &= (g_{1},\dots,g_{s}) \\
        \mathbf{h} &= (h_{1},\dots,h_{s}),
    \end{align*}
    of elements of $ A $ such that $ (\mathbf{g},\mathbf{h}) $ is an $ \alpha $-pair.  The \emph{pedestal ideal} of $ A $ is the ideal generated by zero and all elements $ \prod_{i=1}^{s} h_{i} $ such that there exists a quasi-principle $ \alpha $-pair $ (\mathbf{g},\mathbf{h}) $ where 
    \begin{align*}
        \mathbf{g} &= (g_{1},\dots,g_{s}) \\
        \mathbf{h} &= (h_{1},\dots,h_{s}).
    \end{align*}  
    We denote the large pedestal ideal of $ A $ by $ \mathfrak{P}_{g}(A) $ and the pedestal ideal of $ A $ by $ \mathfrak{P}(A) $.  A point $ x $ is affine stable if $ x \in \operatorname{Spec}(A) \setminus \mathcal{V}(\mathfrak{P}(A)) $.  We shall later show that $ x $ is affine stable if there is a $ G $-stable, open sub-variety $ D(H) $ such that $ D(H) $ is a trivial $ G $-bundle over $ G \backslash \backslash D(H) $.
\end{dfn}
\begin{rmk}
    We should caution the reader that even if $ k $ is a field of characteristic zero and $ G $ is the group $ \mathbb{G}_{a}^{N} $, a pair may not exist if $ N $ is larger than one.  Do not assume that if $ N>1 $ and $ \operatorname{Spec}(A) $ is a $ \mathbb{G}_{a}^{N} $ variety over a field of characteristic zero, then $ \mathfrak{P}(A) \ne 0 $.  Let $ k[t_{1},t_{2}] $ be the affine coordinate ring of $ \mathbb{G}_{a}^{2} $ and let $ \mathbf{V} $ be a vector space of dimension three over a field of characteristic zero with dual basis $ \{x_{1},x_{2},x_{3}\} $.  If $ \beta: \mathbb{G}_{a}^{2} \to \mathbf{V} $ is the representation with the co-action below:
    \begin{align*}
        x_{1} & \mapsto x_{1} \\
        x_{2} & \mapsto x_{2} \\
        x_{3} & \mapsto x_{3}+t_{2}x_{2}+t_{1}x_{1},
    \end{align*}
    then $ \mathfrak{P}(k[X]) $ is equal to zero.
\end{rmk}
\begin{dfn}
    If $ G $ is a connected, unipotent, linear algebraic group, and \linebreak $ \alpha \in \mathbf{P}(G) $, then we obtain a homogeneous space $ G^{\alpha} $ as follows.  The action $ \gamma_{\alpha}: G \times G^{\alpha} $ sends $ (w,s) $ to $ \alpha(w) \cdot s $.
\end{dfn}
\begin{lem} \label{L:alphaPair}
    Let $ G $ be a connected, unipotent, linear algebraic group acting on a variety $ \operatorname{Spec}(A) $ via a left action $ \beta $.  If $ \alpha \in \mathbf{P}(G) $ and $ H \in A^{G} $, then the following are equivalent:
    \begin{itemize}
        \item[a)] there exists a dominant, $ G $-equivariant morphism $ \Phi: D(H) \to G^{\alpha} $,
        \item[b)] there exists an $ \alpha $-pair $ (\mathbf{g},\mathbf{h}) $ such that $ D(H) $ is contained in $ D(\prod_{i=1}^{s} h_{i}) $.
    \end{itemize}
\end{lem}
\begin{proof}
    Assume that a) holds.  If this is true, then there are rational functions $ r_{1},\dots,r_{s} \in A_{H} $ such that $ \Phi(x) $ is equal to $ (r_{1}(x),\dots,r_{s}(x)) $.  For each $ i $ there are functions $ g_{i}, h_{i} \in A_{H} $ such that $ r_{i} = g_{i}/h_{i} $.  Note that $ D(H) $ is contained in $ D(\prod_{i=1}^{s} h_{i}) $, because $ g_{i}/h_{i} \in A_{H} $.  Let $ (\mathbf{g},\mathbf{h}) $ be a pair of $ s $-tuples of $ A $ such that
    \begin{align*}
        \mathbf{g} &= (g_{1},\dots,g_{s}) \\
        \mathbf{h} &= (h_{1},\dots,h_{s}).
    \end{align*}
    We should also note that $ \Phi(x) $ is equal to $ ((g_{1}/h_{1})(x),\dots,(g_{s}/h_{s})(x)) $.  Since $ \Phi $ is $ G $-equivariant, the following diagram commutes:
    \begin{equation*}
    \xymatrix{
        G \times D(H) \ar[rr]^{\beta} \ar[d]^{(\operatorname{id}_{G}, \Phi)} & & D(H) \ar[d]^{\Phi} \\
        G \times G^{\alpha} \ar[rr]^{\gamma_{\alpha}} & & G^{\alpha}
    }.
    \end{equation*}
    As a result,
    \begin{align*}
        \left((g_{1}/h_{1})(w \ast x),\dots,(g_{s}/h_{s})(w \ast x)\right) &= \Phi(w \ast x) \\
        &= \Phi \circ \beta((w, x)) \\
        &= \gamma_{\alpha} \circ (\operatorname{id}_{G}, \Phi)((w,x)) \\
        &= \gamma_{\alpha}\left(w,\left((g_{1}/h_{1})(x),\dots,(g_{s}/h_{s})(x)\right)\right) \\
        &= \alpha(w) \cdot \left((g_{1}/h_{1})(x),\dots,(g_{s}/h_{s})(x)\right).
    \end{align*}
    Therefore $ (\mathbf{g},\mathbf{h}) $ is an $ \alpha $-pair if $ \operatorname{trdeg}_{k}(k[g_{1}/h_{1},\dots,g_{s}/h_{s}]) $ is equal to $ s $.  Because $ \Phi $ is dominant $ \Phi^{\sharp}: k[Y] \to A_{H} $ is injective.  The image of $ \Phi^{\sharp} $ is $ k[g_{1}/h_{1},\dots,g_{s}/h_{s}] $.  Therefore, if $ g_{1}/h_{1},\dots,g_{s}/h_{s} $ are not algebraically independent over $ k $, then $ \Phi^{\sharp} $ is not injective.  As a result, $ \operatorname{trdeg}_{k}(k[g_{1}/h_{1},\dots,g_{s}/h_{s}]) $ is equal to $ s $.  So b) holds.

    If b) holds, then $ (\mathbf{g},\mathbf{h}) $ is an $ \alpha $-pair and $ D(H) $ is contained in $ D(\prod_{i=1}^{s} h_{i}) $.  Let \linebreak $ \Phi: D(H) \to G^{\alpha} $ be the morphism described below:
    \begin{equation*}
        \Phi(x) = \left((g_{1}/h_{1})(x),\dots,(g_{s}/h_{s})(x)\right).
    \end{equation*}
    Because $ (\mathbf{g},\mathbf{h}) $ is an $ \alpha $-pair:
    \begin{align*}
        \Phi \circ \beta(w, x) &= \Phi(w \ast x) \\
        &= \left((g_{1}/h_{1})(w \ast x),\dots,(g_{s}/h_{s})(w \ast x)\right) \\
        &= \alpha(w) \cdot \left((g_{1}/h_{1})(x),\dots,(g_{s}/h_{s})(x)\right) \\
        &= \gamma_{\alpha}\left(w,\left((g_{1}/h_{1})(x),\dots,(g_{s}/h_{s})(x)\right)\right) \\
        &= \gamma_{\alpha} \circ (\operatorname{id}_{G}, \Phi)(w,x).
    \end{align*}
    Therefore, the following diagram commutes:
    \begin{equation*}
    \xymatrix{
        G \times D(H) \ar[rr]^{\beta} \ar[d]^{(\operatorname{id}_{G}, \Phi)} & & D(H) \ar[d]^{\Phi} \\
        G \times G^{\alpha} \ar[rr]^{\gamma_{\alpha}} & & G^{\alpha}
    }.
    \end{equation*}
    So, $ \Phi $ is $ G $-equivariant.  Since $ g_{1}/h_{1},\dots,g_{s}/h_{s} $ are algebraically independent, $ \Phi^{\sharp} $ is an isomorphism of the polynomial ring $ k[Y] $ onto its image $ k[g_{1}/h_{1},\dots,g_{s}/h_{s}] $.  Therefore $ \Phi $ is dominant.
\end{proof}
Let $ H \in A^{G} $.  Later we show an $ \alpha $-pair $ (\mathbf{g},\mathbf{h}) $ such that $ D(H) \subseteq D(\prod_{i=1}^{s} h_{i}) $ exists, if and only if there is a $ G $-equivariant, fppf cover $ \phi: U \to D(H) $ such that 
\begin{itemize}
    \item[i)] for all $ x \in D(H) $, the fibre $ \phi^{-1}(x) = \ker(\alpha) $ and
    \item[ii)] the variety $ U $ is a trivial $ G $-bundle over $ (G \backslash \backslash U) $.
\end{itemize}
\begin{lem} \label{L:connectedNonIdentityGeneral}
    Let $ G $ be a connected, unipotent, linear algebraic group.  If $ \operatorname{Spec}(A) $ is a $ G $-variety with a left $ G $ action, $ \alpha \in \mathbf{P}(G) $, and $ H \in A^{G} $, then $ \Phi: D(H) \to G^{\alpha} $ is a dominant, generically smooth, $ G $-equivariant morphism whose fibres are connected if and only if $ \Phi $ does not factor through a non-automorphism $ \psi \in \mathbf{P}(G) $ and there is no proper, linear algebraic sub-group $ H \le G $ such that $ \Phi $ factors through the inclusion of $ H $ in $ G $.
\end{lem}
\begin{proof}
    One will note that for any $ \psi \in \mathbf{P}(G) $ the kernel of $ \psi $ is finite.  So, if $ \psi $ is generically smooth, then the fibres of $ \psi $ are connected if and only if $ \psi $ an automorphism.  As a result, if $ \Phi $ factors through a separable non-automorphism $ \psi \in \mathbf{P}(G) $, then the fibres of $ \Phi $ are not connected. If the fibres of $ \psi $ are connected and $ \Phi $ factors through a non-automorphism $ \psi $, then $ \psi $ is not separable.  Hence $ \Phi $ is not generically smooth.  As a result, if $ \Phi $ factors through a non-automorphism $ \psi \in \mathbf{P}(G) $, then either $ \Phi $ is not generically smooth, or the fibres of $ \Phi $ are not connected.  Since $ \Phi $ is dominant, it does not factor through any inclusion $ H \le G $ of a proper linear algebraic sub-group $ H $ of $ G $.

    Now suppose that $ \Phi: D(H) \to G^{\alpha} $ is a $ G $-equivariant morphism which does not factor through a non-automorphism $ \psi \in \mathbf{P}(G) $ or an inclusion of a proper sub-group $ H \le G $.  Because $ D(H) $ is affine it embeds as an open sub-variety of a projective variety $ Z $.  Similarly, since $ G $ is a connected, linear algebraic group, it embeds as an open sub-variety of a projective variety $ \mathbb{P}^{s}_{k} $.  We may resolve the indeterminacies of the rational map $ \Phi: Z \dashrightarrow \mathbb{P}^{s}_{k} $ to obtain a morphism $ \Psi: W \to \mathbb{P}^{s}_{k} $.  Here, $ W $ is obtained from $ Z $ via blow-ups and $ \Psi \mid_{D(H)} \cong \Phi $.

    By the Stein factorization theorem there are morphisms $ \psi $ and $ \Psi_{1} $ such that: $ \psi $ is finite, the fibres of $ \Psi_{1} $ are connected, and $ \Psi $ is equal to $ \psi \circ \Phi_{1} $.  Since $ \Phi $ does not factor through the inclusion of any proper, linear algebraic sub-group, $ \Phi $ is dominant.  Hence $ \Psi $ is surjective.  Because $ \psi $ is finite, $ \mathbb{P}^{s}_{k} $ is simply connected and $ \psi \circ \Psi_{1} \mid_{D(H)} \cong \Phi $ is $ G $-equivariant, $ \psi \mid_{\Psi_{1}(D(H))} $ is a surjective endomorphism of $ G $.  An endomorphism $ \psi $ of $ G $ is surjective if and only if $ \psi \in \mathbf{P}(G) $, so $ \psi \in \mathbf{P}(G) $.  By our assumption this morphism is an automorphism.  As a result, the fibres of $ \Phi $ are connected.

    Suppose that $ \Phi $ is not generically smooth.  If this is the case, then the inseparable closure of $ K(G) $ in $ \operatorname{Frac}(A) $ is not equal to $ K(G) $.  Therefore, $ \Phi $ factors through an inseparable endomorphism $ \psi:G \to G $.  Since $ \Phi $ is dominant, the endomorphism $ \psi \in \mathbf{P}(G) $.  This contradicts our assumption that $ \Phi $ does not factor through a non-automorphism of $ \mathbf{P}(G) $.  So $ \Phi $ is generically smooth.  This finishes the proof.
\end{proof}
\begin{lem}
    Let $ G $ be a connected, unipotent, linear algebraic group acting on a variety $ \operatorname{Spec}(A) $ via a left action $ \beta $.  A principle pair exists if and only if an $ \alpha $-pair exists for an automorphism $ \alpha $ of $ G $.
\end{lem}
\begin{proof}
    Suppose an $ \alpha $-pair $ (\mathbf{g},\mathbf{h}) $ exists.  If $ H $ is equal to $ \prod_{i=1}^{s} h_{i} $, then by Lemma ~\ref{L:alphaPair} there exists a dominant, $ G $-equivariant morphism $ \Phi: D(H) \to G^{\alpha} $.  The morphism $ \alpha^{-1} \circ \Phi: D(H) \to G $ is also a dominant, $ G $-equivariant morphism, so a principle pair exists by Lemma ~\ref{L:alphaPair}.
    
    Similarly if a principle pair $ (\mathbf{g},\mathbf{h}) $ exists and $ H $ is equal to $ \prod_{i=1}^{s} h_{i} $, then there exists a dominant, $ G $-equivariant morphism $ \Phi: D(H) \to G $.  The morphism $ \alpha \circ \Phi: D(H) \to G^{\alpha} $ is a dominant, $ G $-equivariant morphism, so an $ \alpha $-pair exists by Lemma ~\ref{L:alphaPair}
\end{proof}
\begin{prop} \label{P:principlePairUnip}
    Let $ G $ be a connected, unipotent, linear algebraic group acting on a variety $ \operatorname{Spec}(A) $ via a left action $ \beta $.  The following statements are equivalent for $ H \in A^{G} $:
    \begin{itemize}
        \item[a)] the sub-variety $ D(H) $ is a trivial $ G $-bundle over $ G \backslash \backslash D(H)$,
        \item[b)] there is a principle pair $ (\mathbf{g}, \mathbf{h}) $ such that $ D(H) \subseteq D(\prod_{i=1}^{s} h_{i}) $
        \item[c)] there is a generically smooth, dominant, $ G $-equivariant morphism \linebreak $ \Phi: D(H) \to G $ whose fibres are connected,
        \item[d)] there are $ s $-tuples $ (\mathbf{g},\mathbf{h}) $ such that if $ K $ is the scheme $ \mathcal{V}(\langle g_{1},\dots,g_{s} \rangle) \cap D(H) $, then the following conditions hold:
            \begin{itemize}
                \item[i)] the stabilizer $ G_{x} $ is trivial for any $ x \in K $,
                \item[ii)] there is an inclusion of open sub-varieties $ D(H) \subseteq D(\prod_{i=1}^{s} h_{i}) $
                \item[iii)] if $ \iota $ is the natural inclusion of $ K $ in $ D(H) $ and $ \pi $ is the quotient morphism from $ D(H) $ to $ G \backslash \backslash D(H) $, then $ \pi \circ \iota $ is an isomorphism.  Namely, $ K $ is a cross section of $ D(H) $.
            \end{itemize}
    \end{itemize}
\end{prop}
\begin{proof}
    If a) holds, then there is an isomorphism $ \lambda :D(H) \to G \times (G \backslash \backslash D(H)) $.  If $ p_{1} $ is the natural projection of $ G \times (G \backslash \backslash D(H)) $ onto the first component, then $ p_{1} \circ \lambda $ is the desired morphism $ \Phi $.  Both $ p_{1} $ and $ \lambda $ are $ G $-equivariant, dominant, and generically smooth.  Also, because $ G $, $ D(H) $ and $ G \backslash \backslash D(H) $ are connected, the fibres of $ p_{2} $ and $ \lambda $ are connected.  Therefore c) holds.  So a) implies c).

    Assume that c) holds.  Let $ Y $ be the scheme $ \Phi^{-1}(e) $ with its reduced induced sub-scheme structure.  The scheme $ Y $ is a variety by Lemma ~\ref{L:connectedNonIdentityGeneral}.  Define a morphism $ \gamma: G \times Y \to D(H) $ which sends $ (w,y) $ to $ w \ast y $.  Suppose that there are two distinct points $ (g_{1},y_{1}) $ and $ (g_{2},y_{2}) $ of $ G \times Y $, such that $ g_{1} \ast y_{1} = g_{2} \ast y_{2} $.  If this is true, then $ y_{1} = (g_{1}^{-1} \cdot g_{2}) \ast y_{2} $.  Because $ Y $ is the scheme theoretic pre-image of $ e $ with its reduced induced scheme structure,
    \begin{align*}
        g_{1}^{-1} \cdot g_{2} &= g_{1}^{-1} \cdot g_{2} \cdot \Phi(y_{2}) \\
        &= \Phi(g_{1}^{-1} \cdot g_{2} \ast y_{2}) \\
        &= \Phi(y_{1}) \\
        &= e.
    \end{align*}
    However, this means that $ g_{1}=g_{2} $ and therefore that $ y_{1}=y_{2} $.  This contradicts our assumption that the two points $ (g_{1},y_{1}) $ and $ (g_{2},y_{2}) $ were distinct.  So $ \gamma $ is injective.  Since $ \Phi $ is generically smooth,
    \begin{align*}
        \dim(A)=s+\dim(Y).
    \end{align*}
    So $ \gamma $ is dominant.  Because the image of $ \gamma $ is closed, $ \gamma $ is dense and closed.  So $ \gamma $ is surjective.  If $ \lambda: D(H) \to G \times Y $ is the morphism which sends a point $ x $ to $ (\Phi(x), \Phi(x)^{-1} \ast x) $, then $ \lambda $ is an inverse to $ \gamma $.  Since $ \Phi $ is generically smooth, so too is $ \lambda $.  Therefore, $ \lambda $ is an isomorphism.  A categorical quotient is unique, so $ Y \cong G \backslash \backslash D(H) $.  So c) implies a).  Therefore c) and a) are equivalent.

    Assume that c) holds.  There are $ s $ rational functions $ r_{1},\dots,r_{s} \in A_{H} $ such that $ \Phi(x) $ is equal to $ (r_{1}(x),\dots,r_{s}(x)) $.  For each $ r_{i} \in A_{H} $ there is a pair $ g_{i},h_{i} $ such that $ r_{i} $ is equal to $ g_{i}/h_{i} $.  Therefore, there are two $ s $-tuples $ (\mathbf{g},\mathbf{h}) $ equal to
    \begin{align*}
        \mathbf{g} &= (g_{1},\dots,g_{s}) \\
        \mathbf{h} &= (h_{1},\dots,h_{s}) \\
    \end{align*}
    such that $ \Phi(x) $ is equal to $ ((g_{1}/h_{1})(x),\dots,(g_{s}/h_{s})(x)) $.  Because $ \Phi $ is $ G $-equivariant, the following diagram commutes:
    \begin{equation*}
    \xymatrix{
        G \times D(H) \ar[rr]^{\beta} \ar[d]^{(\operatorname{id}_{G}, \Phi)} & & D(H) \ar[d]^{\Phi} \\
        G \times G \ar[rr]^{\mu_{G}} & & G
    }.
    \end{equation*}
    So, for $ (w,x) \in G \times D(H) $
    \begin{align*}
        \left((g_{1}/h_{1})(w \ast x),\dots,(g_{s}/h_{s})(w \ast x)\right) &=\Phi(w \ast x) \\
        &= \Phi \circ \beta(w,x) \\
        &= \mu_{G} \circ (\operatorname{id}_{G}, \Phi)(w,x) \\
        &= \mu_{G}\left(w,\left((g_{1}/h_{1})(x),\dots,(g_{s}/h_{s})(x)\right)\right) \\
        &= w \cdot \left((g_{1}/h_{1})(x),\dots,(g_{s}/h_{s})(x)\right).
    \end{align*}
    Therefore, if the transcendence degree of $ \{g_{i}/h_{i}\}_{i=1}^{s} $ is $ s $, then $ (\mathbf{g},\mathbf{h}) $ is a principle pair.  Since $ \Phi $ is dominant $ \Phi^{\sharp}: k[Y] \to A_{H} $ is injective.  The image of $ k[Y] $ under $ \Phi^{\sharp} $ is $ k[g_{1}/h_{1},\dots,g_{s}/h_{s}] $.  If $ \{g_{i}/h_{i}\}_{i=1}^{s} $ is not algebraically independent over $ k $, then $ \Phi^{\sharp} $ is not injective.  So $ (\mathbf{g},\mathbf{h}) $ is a principle pair and c) implies b).

    Assume that b) holds, i.e., a principle pair $ (\mathbf{g},\mathbf{h}) $ exists and $ D(H) \subseteq D(\prod_{i=1}^{s} h_{i}) $.  Define $ \Phi: D(H) \to G $ to be the morphism below
    \begin{equation*}
        \Phi(x)=\left((g_{1}/h_{1})(x),\dots,(g_{s}/h_{s})(x)\right).
    \end{equation*}
    We claim that $ \Phi $ is $ G $-equivariant.  Because $ (\mathbf{g},\mathbf{h}) $ is a principle pair:
    \begin{align*}
        \Phi \circ \beta(w,x) &=\Phi(w \ast x) \\
        &= \left((g_{1}/h_{1})(w \ast x),\dots,(g_{s}/h_{s})(w \ast x)\right) \\
        &= w \cdot\left((g_{1}/h_{1})(x),\dots,(g_{s}/h_{s})(x)\right) \\
        &= \mu_{G}\left(w,\left((g_{1}/h_{1}(x)),\dots,(g_{s}/h_{s})(x)\right)\right) \\
        &= \mu_{G} \circ (\operatorname{id}_{G}, \Phi)(w,x).
    \end{align*}
    Therefore, the following diagram commutes:
    \begin{equation*}
    \xymatrix{
        G \times D(H) \ar[rr]^{\beta} \ar[d]^{(\operatorname{id}_{G}, \Phi)} & & D(H) \ar[d]^{\Phi} \\
        G \times G \ar[rr]^{\mu_{G}} & & G
    }.
    \end{equation*}
    So $ \Phi $ is $ G $-equivariant.  Since $ \Phi^{\sharp} $ maps the affine coordinate ring of $ G $ isomorphically onto $ k[g_{1}/h_{1},\dots,g_{s}/h_{s}] $ and $ \{g_{i}/h_{i}\}_{i=1}^{s} $ is algebraically independent, $ \Phi $ is dominant.  So $ \Phi $ does not factor through an inclusion of a proper, linear algebraic sub-group $ H $ of $ G $.  Suppose that $ \Phi $ factors through a non-automorphism $ \psi \in \mathbf{P}(G) $.  If this is the case, then there is a $ G $-equivariant, dominant morphism $ \Phi_{1}: D(H) \to G^{\alpha} $ for $ \alpha \in \mathbf{P}(G) $ such that $ \Phi = \psi \circ \Phi_{1} $.  So $ \psi \circ \alpha = \operatorname{id} $ contradicting our assumption that $ \psi $ is not an automorphism.  By Lemma ~\ref{L:connectedNonIdentityGeneral} the morphism $ \Phi $ is generically smooth, dominant and its fibres are connected.  So b) is equivalent to c).  Therefore a), b) and c) are equivalent.
    
    If a) holds, then there is an isomorphism $ \lambda: D(H) \to G \times (G \backslash \backslash D(H)) $.  The equivalence of a), b), and c) shows that if $ \Phi $ is the corresponding, $ G $-equivariant morphism $ p_{2} \circ \lambda $, which in turn corresponds to the principle pair $ (\mathbf{g},\mathbf{h}) $ such that $ D(H) \subseteq D(\prod_{i=1}^{s} h_{i}) $, then
    \begin{align*}
        \mathcal{V}(\langle g_{1},\dots,g_{s} \rangle) \cap D(H) &\cong \Phi^{-1}(e) \\
        &\cong G \backslash \backslash D(H).
    \end{align*}
    So $ \pi \circ \iota $ is an isomorphism.  If we denote the variety below by $ K $:
    \begin{equation*}
        K:= \mathcal{V}(\langle g_{1},\dots,g_{s}, \rangle) \cap D(H),
    \end{equation*}
    then $ \beta $ maps $ G \times K $ to $ D(H) $.  If $ \phi $ is the inverse of $ \pi \circ \iota $, then $ \phi \circ \pi $ is a morphism from $ G $ to $ K $ which is constant on orbits.  So if $ p_{2} $ is the projection of $ G \times K $ onto $ K $, then the following diagram commutes:
    \begin{equation*}
    \xymatrix{
        G \times K \ar[drr]^{\phi^{-1} \circ p_{2}} \ar[rr]^{\beta} & & D(H) \ar[d]^{\pi} \\
        & & G \backslash \backslash D(H)
        }.
    \end{equation*}
    As a result, $ \beta: G \times K \to D(H) $ is surjective.  If $ \Lambda $ is the sub-scheme of $ G \times K $ such that for $ G_{x} \cong \Lambda_{x} $ for any $ x \in K $, then $ \beta $ factors through $ (G \times K)/\Lambda $.  Therefore, there is a morphism $ \nu: (G \times K)/\Lambda \to D(H) $.  We should note that the quotient $ (G \times K)/\Lambda $ may not exist in the category of schemes, but does exist in the category of algebraic spaces.  Since $ \beta $ is separable, so too is $ \nu $.  If $ y \in D(H) $ and $ x $ is a point of $ K $ equal to $ \phi \circ \pi(y) $, then suppose there are $ g_{1},g_{2} \in G $ such that
    \begin{align*}
        g_{1} \ast x &= y \\
        &= g_{2} \ast x.
    \end{align*}
    If this is the case, then $ (g_{2}^{-1} \ast g_{1},x) \in \Lambda $.  Therefore $ \nu(g_{1},x)=\nu(g_{2},x) $.  Hence $ \nu $ is an isomorphism.  Since
    \begin{align*}
        G \times K & \cong G \times \left(G \backslash \backslash D(H)\right) \\
        &\cong D(H) \\
        & \cong (G \times K)/\Lambda
    \end{align*}
    if $ x \in K $, the fibre $ \Lambda_{x} \cong G_{x} $ is equal to the identity.  Therefore d) holds.
    
    If d) holds, then the work we did above in proving $ D(H) \cong (G \times K)/\Lambda $ still holds.  Since $ \Lambda_{y} = e $ for any $ y \in K $,
    \begin{align*}
        D(H) &\cong (G \times K)/\Lambda \\
        & \cong (G \times K)/(e \times K) \\
        & \cong G \times K.
    \end{align*}
    Therefore a) holds.
\end{proof}

One may find the notion of principle pairs intuitive, since it is intimately related to the triviality of $ D(H) $ as a $ G $-bundle over $ G \backslash \backslash D(H) $.  However, the notion of an $ \alpha $-pair may seem un-intuitive since $ D(H) $ is not a trivial $ G $-bundle.  However, in the fppf topology, $ D(H) $ is a trivial $ G $-bundle.  Namely there is an fppf neighborhood of $ D(H) $ which is a trivial $ G $-bundle.
\begin{prop}
    Let $ G $ be a connected, unipotent, linear algebraic group acting on a variety $ \operatorname{Spec}(A) $ via a left action $ \beta $.  If $ H \in A^{G} $, then the following are equivalent:
    \begin{itemize}
        \item[i)] there is an $ \alpha $-pair $ (\mathbf{g},\mathbf{h}) $ such that $ D(H) \subseteq D(\prod_{i=1}^{s} h_{i}) $,
        \item[ii)] there is a $ G $-equivariant, fppf cover $ \phi: U \to D(H) $ such that $ U $ is a trivial $ G $ bundle over $ G \backslash \backslash U $ and the fibres $ \phi^{-1}(x) \cong \operatorname{ker}(\alpha) $ for $ x \in D(H) $.
    \end{itemize}
\end{prop}
\begin{proof}
    If i) holds, then define an action $ \gamma $ of $ G $ on $ G \times D(H) $ as follows.  For $ w \in G $, and $ (u,x) \in G \times D(H) $, let $ \gamma(w,(u,x)) $ equal $ (w \cdot u, w \ast x) $.  The existence of an $ \alpha $-pair $ (\mathbf{g},\mathbf{h}) $ such that $ D(H) \subseteq D(\prod_{i=1}^{s} h_{i}) $ is equivalent to the fact that the morphism $ \Phi: D(H) \to G^{\alpha} $ which sends $ x $ to $ \left((g_{1}/h_{1})(x),\dots,(g_{s}/h_{s})(x)\right) $ is dominant and $ G $-equivariant by Lemma ~\ref{L:alphaPair}.  As a result, the morphism $ (\alpha,\Phi): G \times D(H) \to G^{\alpha} \times G^{\alpha} $ is $ G $-equivariant.

    The morphism $ (\alpha,\alpha) : G \times G \to G^{\alpha} \times G^{\alpha} $ is $ G $-equivariant.  So, if $ \Delta_{G} $ is the diagonal sub-scheme of $ G \times G $, then $ (\alpha,\alpha): \Delta_{G} \to G^{\alpha} \times G^{\alpha} $.  Let $ U $ be the product $ \Delta_{G} \times_{(G^{\alpha} \times G^{\alpha})} \left(G \times D(H)\right) $.  If $ A_{H}[W] $ is the affine coordinate ring of $ G \times D(H) $, then because $ U $ is isomorphic to $ \Delta_{G} \times_{(G^{\alpha} \times G^{\alpha})}\left(G \times D(H)\right) $,
    \begin{equation} \label{E:96}
        U \cong \mathcal{V}(\langle \alpha^{\sharp}(w_{i})-g_{i}/h_{i}\rangle_{i=1}^{s}).
    \end{equation}
    Equation ~\eqref{E:96} shows that $ \left((w_{1},\dots,w_{s}),(1,\dots,1)\right) $ is a principle pair for the $ G $-action on $ U $.  By Proposition ~\ref{P:principlePairUnip}, the variety $ U $ is a trivial $ G $-bundle over $ G \backslash \backslash U $.

    The following diagram commutes:
    \begin{equation} \label{E:100}
    \xymatrix{
        U \ar[rr]^{p_{2}} \ar[d]^{p_{1}} & & G \times D(H) \ar[d]^{(\alpha,\Phi)} \\
        \Delta_{G} \ar[rr]^{(\alpha,\alpha)} & & G^{\alpha} \times G^{\alpha}
        }.
    \end{equation}
    Let $ \phi: U \to D(H) $ be the natural morphism obtained from the ring homomorphism $ \phi^{\sharp}: A_{H} \to A_{H}[W]/\langle \alpha^{\sharp}(w_{i})-g_{i}/h_{i} \rangle $.  Because $ U $ is a closed sub-scheme of $ G \times D(H) $, the projection which sends $ (w,x) \to x $ is $ \phi $.  So, if $ (w_{1},x) \in \phi^{-1}(x) $, then the diagram in ~\eqref{E:100} shows that $ (\alpha,\Phi)(w_{1},x) \in (\alpha,\alpha)(\Delta_{G}) $.  Therefore, if $ (w_{2},x) $ is another point of $ \phi^{-1}(x) $, then there is a $ g \in G $ such that $ \Phi(x) = g $ and
    \begin{align*}
        (\alpha(w_{1}), \Phi(x)) &= (\alpha,\Phi)(w_{1},x) \\
        &= (g,g) \\
        &= (\alpha, \Phi)(w_{2},x) \\
        &= (\alpha(w_{2}),\Phi(x)).
    \end{align*}
    So $ \alpha(w_{1})=\alpha(w_{2}) $ if $ (w_{1},x),(w_{2},x) \in \phi^{-1}(x) $.  Therefore $ w_{1} \cdot w_{2}^{-1} \in \ker(\alpha) $.  So $ \phi^{-1}(x) \cong \ker(\alpha) $ for all $ x \in D(H) $.  
    
    Because $ G \cong \mathbb{A}^{s}_{k} $, we may embed $ G $ as the open sub-variety
    \begin{align*}
        G & \cong D_{+}(u_{0}) \\
        & \subseteq \operatorname{Proj}(k[u_{0},\dots,u_{s}]).
    \end{align*}
    Let $ Z $ be the closure of $ U $ in $ \mathbb{P}^{s}_{k} \times D(H) $, and let $ \psi: Z \to D(H) $ be the projection onto the second component.  For any point $ x \in D(H) $, the fibre $ \psi^{-1}(x) $ is isomorphic to the closure of $ \ker(\alpha) $ in $ \mathbb{P}^{s}_{k} $.  As a result, the Hilbert polynomial of $ Z_{x} $ does not depend on the point $ x \in D(H) $.  So by \cite[Chapter III, Cohomology, Section 9, Flat Morphisms, Theorem 9.9]{HartshorneAG}, the scheme $ Z $ is flat over $ D(H) $.  This means that $ \phi: U \to D(H) $ is also flat.  A local, flat ring homomorphism is faithfully flat, so $ \phi^{\sharp}_{(w,x)} $ is faithfully flat for each $ (w,x) \in U $.  As a result, $ \phi $ is fppf. Therefore ii) holds.

    Suppose ii) holds.  Let $ y $ be a point in $ \phi^{-1}(x) $, where $ \phi: U \to D(H) $ is the fppf neighborhood mentioned in ii).  Since $ U \cong G \times \left(G \backslash \backslash U\right) $ there is an isomorphism $ \lambda: U \to G \times \left( G \backslash \backslash U\right) $.  If $ p_{1} $ is the projection of $ G \times \left( G \backslash \backslash U \right) $ onto the first component, then $ p_{1} \circ \lambda $ is a dominant, $ G $-equivariant, generically smooth morphism $ \Psi $ whose fibres are connected.  If $ y $ is equal to $ (w,x) $ and $ y_{1} $ is another point in $ \phi^{-1}(x) $ equal to $ (w_{1},x) $ then $ w \cdot w_{1}^{-1} \in \ker(\alpha) $ by our assumption that $ \phi^{-1}(x) \cong \ker(\alpha) $.  As a result, if $ \Phi(x) $ is the map which sends $ x $ to $ \alpha(\Psi(y)) $ for any $ y \in \phi^{-1}(x) $, then $ \Phi $ does not depend on the choice of $ y \in \phi^{-1}(x) $ but only on $ x $.  So $ \Phi $ is well defined.  Note that the following diagram commutes:
    \begin{equation*}
    \xymatrix{
        U \ar[d]^{\phi} \ar[rr]^{\Psi} & & G \ar[d]^{\alpha} \\
        D(H) \ar[rr]^{\Phi} & & G^{\alpha}
        }.
    \end{equation*}
    Therefore, since $ \phi, \alpha $ and $ \Psi $ are dominant and $ G $-equivariant, so too is $ \Phi $.  By Lemma ~\ref{L:alphaPair}, item i) holds.
\end{proof}
As a result, the existence of an $ \alpha $-pair is intimately related to local triviality in the fppf topology.
\begin{thm}
    Let $ G $ be a connected, unipotent, linear algebraic group acting on a variety $ \operatorname{Spec}(A) $ via a left action $ \beta $.  If $ H \in A^{G} $, then the following are equivalent:
    \begin{itemize}
        \item[a)] there is a quasi-principle $ \alpha $-pair $ (\mathbf{g},\mathbf{h}) $ such that $ D(H) \subseteq D(\prod_{i=1}^{s} h_{i}) $,
        \item[b)] the variety $ D(H) \cong \left((G/\ker(\alpha)) \times (G \backslash \backslash D(H))\right) $,
        \item[c)] the kernel of $ \alpha $ acts trivially upon $ D(H) $ and there is a generically smooth, dominant, $ G $-equivariant, morphism $ \Phi: D(H) \to G^{\alpha} $ whose fibres are connected.
    \end{itemize}
\end{thm}
\begin{proof}
    Upon replacing $ G $ with $ G/\ker(\alpha) $ this theorem becomes parts a), b) and c) of Proposition ~\ref{P:principlePairUnip}.
\end{proof}
\begin{thm}\label{T:dixmierUnip}
    Let $ G $ be a connected, unipotent, linear algebraic group acting on a variety $ \operatorname{Spec}(A) $ via a left action $ \beta $.
    
    If $ (\mathbf{g},\mathbf{h}) $ is a principle pair, $ A $ is equal to $ k[z_{1},\dots,z_{n}] $ as a $ k $-algebra, $ H $ is equal to $ \prod_{i=1}^{s} h_{i} $, and $ \beta^{\sharp}(z_{i}) $ is equal to $ v_{i}(Z,Y) $, then there are elements $ b_{1},\dots,b_{s} \in A_{H} $ such that the point $ \left((g_{1}/h_{1})(x),\dots,(g_{s}/h_{s})(x)\right)^{-1} \in G $ is equal to $ (b_{1}(x),\dots,b_{s}(x)) $.  If $ f_{i}(Z) $ is equal to $ v_{i}(Z,b_{1},\dots,b_{s}) $, then $ f_{i}(Z) \in A_{H}^{G} $ and
    \begin{equation*}
        A_{H}^{G}= k[f_{1}(Z),\dots,f_{n}(Z)].
    \end{equation*}
\end{thm}
\begin{proof}
    If $ w_{1},w_{2} \in G $ and $ x \in D(H) $, then
    \begin{align}
        v_{i}(w_{2} \ast x,w_{1}) &= z_{i}(w_{1} \ast (w_{2} \ast x)), \notag \\
        &= z_{i}((w_{1} \cdot w_{2}) \ast x), \notag \\
        &= v_{i}(x,w_{1} \cdot w_{2}). \label{E:3}
    \end{align}
    Equating the left and right hand sides of ~\eqref{E:3} we obtain
    \begin{equation} \label{E:4}
        v_{i}(w_{2} \ast x ,w_{1}) = v_{i}(x,w_{1} \cdot w_{2}).
    \end{equation}
    If $ w \in G $ and $ x \in D(H) $, then set $ w_{1} $ and $ w_{2} $ to be the following points:
    \begin{align*}
        w_{1} &= (b_{1}(w \ast x),\dots,b_{s}(w \ast w)) \\
        w_{2} &= w.
    \end{align*}
    The following calculations show that $ f_{i}(Z) $ is constant on the orbits of points \linebreak $ x \in D(H) $,
    \begin{align*}
        f_{i}(w \ast x) &= v_{i}(w \ast x, b_{1}(w \ast x),\dots,b_{s}(w \ast x)) \\
        &=v_{i}(w \ast x, b_{1}(w \ast x),\dots,b_{s}(w \ast x)) \\
        &=v_{i}\left(w \ast x, \left((g_{1}/h_{1})(w \ast x),\dots (g_{s}/h_{s})(w \ast x)\right)^{-1} \right) \\
        &= v_{i}\left(w \ast x, \left(w \cdot (g_{1}/h_{1})(x),\dots (g_{s}/h_{s})(x)\right)^{-1}  \right) \\
        &= v_{i}(w \ast x, \left((g_{1}/h_{1})(x),\dots,(g_{s}/h_{s})(x)\right)^{-1} \cdot w^{-1} ) \\
        &= v_{i}(x, \left((g_{1}/h_{1})(x),\dots,(g_{s}/h_{s})(x)\right)^{-1} \cdot w^{-1} \cdot w ) \\
        &= v_{i}(x, \left((g_{1}/h_{1})(x),\dots,(g_{s}/h_{s})(x)\right)^{-1}) \\
        &= v_{i}(x, b_{1}(x),\dots,b_{s}(x)) \\
        &= f_{i}(x),
    \end{align*}
    where the jump from fifth line to the sixth line uses ~\eqref{E:4}.  Since $ f_{i}(Z) $ is invariant on the orbits of $ G $, it is in $ A_{H}^{G} $.

    If $ r(Z) \in A_{H}^{G} $ and $ k[Y] $ is the affine coordinate ring of $ G $, then
    \begin{align}
        r(Z) &= \beta^{\sharp}(r(Z)), \notag \\
        &= r(v_{1}(Z,Y),\dots,v_{n}(Z,Y)) \label{E:5}
    \end{align}
    The left hand side of ~\eqref{E:5} does not depend on $ y_{1},\dots,y_{s} $, so we may substitute $ b_{i} $ for $ y_{i} $ to obtain
    \begin{equation*}
        r(Z) = r(f_{1}(Z),\dots,f_{n}(Z)).
    \end{equation*}
    Therefore $ A_{H}^{G} $ is equal to $ k[f_{1}(Z),\dots,f_{n}(Z)] $.
\end{proof}
\section{Open Questions.}
The following are open questions that I would like to leave for the reader.  The first is regarding the pair monoid $ \mathbf{P}(G) $.  If $ G $ is the ring $ \mathbb{G}_{a} $, then any surjective endomorphism $ \psi: \mathbb{G}_{a} \to \mathbb{G}_{a} $ is determined by an additive polynomial $ c(t) $.  These endomorphisms correspond to elements of the Ore ring $ k[F] $ where addition is pointwise addition, but multiplication is composition.  Namely $ aF^{i}(bF^{j}) $ is equal to $ ab^{p^{i}}F^{i+j} $.  This does have a generalization when the linear algebraic group $ G $ is a connected, unipotent, linear algebraic group.
\begin{dfn}
    Let $ G $ be an $ s $-dimensional, connected, linear algebraic group.  The set $ \mathfrak{O}(G) $ is the $ k $-algebra composed of ring homomorphisms $ \phi: k[Y] \to k[X] $ such that $ \mu_{G}^{\sharp} \circ \phi= (\phi \otimes \phi) \circ \mu_{G}^{\sharp} $ and $ \epsilon_{G}^{\sharp} \circ \phi = \epsilon_{G}^{\sharp} $.
\end{dfn}
The set $ \mathfrak{O}(\mathbb{G}_{a}) $ may be endowed with the structure of a ring.  However, we cannot a'priori guarantee that $ \mathfrak{O}(G) $ has the structure of a ring since it is not clear if $ \mathfrak{O}(G) $, is closed under addition.
\begin{lem}
    The set $ \mathfrak{O}(G) $ is a monoid with respect to composition.
\end{lem}
\begin{proof}
    If $ \phi_{1},\phi_{2} \in \mathfrak{O}(G) $, then by the definition of $ \mathfrak{O}(G) $:
    \begin{align*}
        \mu_{G}^{\sharp} \circ (\phi_{1} \circ \phi_{2}) &= (\phi_{1} \otimes \phi_{1}) \circ \mu_{G}^{\sharp} \circ \phi_{2} \\
        &= (\phi_{1}\otimes \phi_{1}) \circ (\phi_{2} \otimes \phi_{2}) \circ \mu_{G}^{\sharp} \\
        &= (\phi_{1} \circ \phi_{2} \otimes \phi_{1} \circ \phi_{2}) \circ \mu_{G}^{\sharp} \\
        \epsilon_{G}^{\sharp} \circ (\phi_{1} \circ \phi_{2}) &= \epsilon_{G}^{\sharp} \circ \phi_{2} \\
        &= \epsilon_{G}^{\sharp}.
    \end{align*}
    Therefore $ \phi_{1} \circ \phi_{2} \in \mathfrak{O}(G) $.  Composition is associative and the identity is an element of $ \mathfrak{O}(G) $, so $ \mathfrak{O}(G) $ is a monoid.
\end{proof}
If $ \mathfrak{O}(G) $ was closed under addition, it would have the structure of a ring.

\emph{Question 1} Is $ \mathfrak{O}(G) $ closed under addition for all connected, unipotent, linear algebraic groups $ G $?

If the answer to Question 1 is yes, then the objects we are really concerned about are injective elements of $ \mathfrak{O}(G) $.  In the case of the Ore ring, the only non-injective element is zero.  This may be thought of as the contraction of $ \mathbb{G}_{a} $ to a point.  However, there are many elements of $ \mathfrak{O}(G) $ that are non-injective for connected, unipotent, linear algebraic groups of dimension greater than one.  What structure do the injective elements have?  Also, the Ore ring has the structure of a non-commutative Euclidean ideal domain.

\emph{Question 2} If the answer to question one is yes, then is there some generalization of the division algorithm to $ \mathfrak{O}(G) $ when $ G $ is a connected, unipotent, linear algebraic group of dimension greater than one?  Such a generalization might involve monomial orders and/or Groebner bases if it exists.

In \cite{Maguire} the author classified all $ \mathbb{G}_{a} $-representations $ \beta: \mathbb{G}_{a} \to \operatorname{GL}(\mathbf{V}) $ such that $ \mathfrak{P}_{g}(S_{k}(\mathbf{V}^{\ast})) $ is equal to zero.

\emph{Question 3} Describe representations $ \beta: G \to \operatorname{GL}(\mathbf{V}) $ of connected, unipotent, linear algebraic groups $ G $ such that $ \mathfrak{P}_{g}(S_{k}(\mathbf{V}^{\ast})) $ is equal to zero.  Are there invariants $ H $ such that it is possible to compute the ring $ S_{k}(\mathbf{V}^{\ast})^{G}_{H} $?  If so, describe how to compute the localization of the ring of invariants.

\emph{Question 4} This question the same as question 3, but for representations connected, unipotent, linear algebraic groups such that the corresponding large pedestal ideal is non-zero, but the pedestal ideal is zero.
\section{Appendix: An Aside on Reductive Groups.}
\begin{dfn}
    The variety $ \mathbb{G}_{m} $ has multiple $ \mathbb{G}_{m} $ actions on it.  One is the action which sends a point $ (\xi,z) $ to $ \xi z $.  However, there is also the action which sends $ (\xi,z) $ to $ \xi^{q} z $.  We shall denote the variety $ \mathbb{G}_{m} $ with this action by $ \left(\mathbb{G}_{m} \right)^{q} $ and the action by $ \gamma_{q} $.  Also recall that if $ q \in \mathbb{N} $, then we obtain an endomorphism of $ \mathbb{G}_{m} $ from the map which sends $ y $ to $ y^{q} $.
\end{dfn}
\begin{prop} \label{P:generalBundleRed}
    Let $ \operatorname{Spec}(A) $ be a $ \mathbb{G}_{m} $-variety with action $ \beta $.  If $ D(h) $ is a $ \mathbb{G}_{m} $-stable sub-variety of $ \operatorname{Spec}(A) $ and $ q \in \mathbb{N} $, then the following are equivalent:
    \begin{itemize}
        \item[a)] there is a dominant, $ \mathbb{G}_{m} $-equivariant morphism $ \Phi:D(h) \to (\mathbb{G}_{m})^{q} $ such that the following diagram commutes:
            \begin{equation} \label{E:24}
            \xymatrix{
                \mathbb{G}_{m} \times D(h) \ar[d]^{(\operatorname{id}_{\mathbb{G}_{m}}, \Phi)} \ar[rr]^{\beta} & & D(h) \ar[d]^{\Phi} \\
                \mathbb{G}_{m} \times \left(\mathbb{G}_{m}\right)^{q} \ar[rr]^{\gamma_{q}} & & \left(\mathbb{G}_{m}\right)^{q}
                }.
            \end{equation}
        \item[b)] there is an $ e \in \mathbb{N}_{0} $ and $ g \in A_{h}^{\ast} $ such that $ g/h^{e} $ is a semi-invariant of weight $ q $, i.e.,
            \begin{equation*}
                (g/h^{e})(w \bullet x) = w^{q} (g/h^{e})(x),
            \end{equation*}
            for all $ w \in \mathbb{G}_{m} $ and $ x \in D(h) $.
    \end{itemize}
\end{prop}
\begin{proof}
    If a) holds, then there is some $ g \in A $ and $ e \in \mathbb{N} $, such that $ \Phi $ sends $ x \in D(h) $ to $ (g/h^{e})(x) $.  Because the diagram in ~\eqref{E:24} commutes, if $ x \in D(h) $ and $ \xi \in \mathbb{G}_{m} $, then
    \begin{align*}
        \left(g/h^{e}\right)(\xi \bullet x) &= \Phi(\xi \bullet x) \\
        &=\Phi \circ \beta(\xi,x) \\
        &=\gamma_{q} \circ (\operatorname{id}_{\mathbb{G}_{m}},\Phi)(\xi,x) \\
        &=\gamma_{q}\left(\xi,(g/h^{e})(x) \right) \\
        &= \xi^{q} (g/h^{e})(x).
    \end{align*}
    As a result, b) holds.

    If b) holds, then define a morphism $ \Phi:D(h) \to (\mathbb{G}_{m})^{q} $ by sending $ x \in D(h) $ to $ (g/h^{e})(x) $.  If $ x \in D(h) $ and $ \xi \in \mathbb{G}_{m} $, then since $ g/h^{e} $ is a semi-invariant of weight $ q $:
    \begin{align}
        \Phi \circ \beta(\xi, x) &= \Phi(\xi \bullet x), \notag \\
        &= (g/h^{e})(\xi \bullet x), \notag \\
        &=  \xi^{q} (g/h^{e})(x), \notag \\
        &= \gamma_{q}(\xi, (g/h^{e})(x)), \notag \\
        &= \gamma_{q} \circ (\operatorname{id}_{\mathbb{G}_{m}}, \Phi)(\xi,x) . \label{E:8}
    \end{align}
    Equating both sides of ~\eqref{E:8} shows that the diagram in ~\eqref{E:24} commutes.  Since $ (g/h^{e}) $ is a semi-invariant of weight $ q \in \mathbb{N} $, it is non-constant.  Therefore, $ \Phi^{\sharp} $ maps $ k[z,z^{-1}] $ isomorphically onto $ k[g/h^{e},h^{e}/g] $.  So, $ \Phi^{\sharp} $ is injective.  Hence $ \Phi $ is dominant.
\end{proof}
\begin{lem} \label{L:connectedNonIdentityRed}
    If $ \ell $ is a natural number, $ \operatorname{Spec}(A) $ is a $ \mathbb{G}_{m} $-variety, and $ D(h) $ is a $ \mathbb{G}_{m} $-stable neighborhood, then a $ \mathbb{G}_{m} $-equivariant morphism $ \Psi: D(h) \to (\mathbb{G}_{m})^{\ell} $ is dominant, generically smooth and its fibres are connected if and only if $ \Psi $ does not factor through any non-identity endomorphism $ q: \mathbb{G}_{m} \to \mathbb{G}_{m} $ and if $ \Psi $ does not contract $ D(h) $ to a point.
\end{lem}
\begin{proof}
    Suppose that $ \Psi $ factors through a non-identity endomorphism $ q $ of $ \mathbb{G}_{m} $.  If $ q $ is equal to $ p^{\ell} d $ where $ p \nmid d $, then the following diagram of field extensions commutes:
    \begin{equation*}
    \xymatrix{
        \operatorname{Frac}(A) \ar@{-}[d] \\
        k(z) \cong K(\mathbb{G}_{m}) \ar@{-}[d] \\
        k(z^{p^{\ell}}) \ar@{-}[d] \\
        k(z^{q}) \cong K(\mathbb{G}_{m}/\mu_{q})
        }.
    \end{equation*}
    If $ \ell >0 $, then $ \Psi $ is not generically smooth.  Suppose that $ \ell $ is equal to zero, but $ d $ is not.  If this is the case, then the fibre of any closed point of $ (\mathbb{G}_{m})^{\ell} $ under $ \Psi $ is the union of $ d $ connected irreducible components.  Therefore, the fibres of $ \Psi $ are not connected.  If $ \Psi $ contracts $ D(h) $ to a point, then $ \Psi $ is not dominant.  What we have shown is the contrapositive form of the statement that if $ \Psi $ is a generically smooth, dominant $ \mathbb{G}_{m} $-equivariant morphism from $ D(h) $ to $ (\mathbb{G}_{m})^{\ell} $ whose fibres are connected, then $ \Psi $ does not factor through any non-identity endomorphism $ q: \mathbb{G}_{m} \to \mathbb{G}_{m} $ and $ \Psi $ does not contract $ D(h) $ to a point.  

    Assume that $ \Psi $ does not factor through a non-trivial endomorphism $ q $ of $ \mathbb{G}_{m} $ and does not contract $ D(h) $ to a point.  Since $ \operatorname{Spec}(A) $ is affine, it is an open sub-variety of a projective variety $ Z $.  We may resolve the indeterminacies of the rational map $ \Psi: Z \dashrightarrow \mathbb{P}^{1}_{k} $ to obtain a morphism $ \Phi: W \to \mathbb{P}^{1}_{k} $.  Here $ W $ is obtained from $ Z $ via blow-ups and $ \Phi \mid_{D(h)} \cong \Psi $.

    By the Stein factorization theorem \cite[III, Cohomomology, Section 11, The Theorem on Formal Functions, Corollary 11.5]{HartshorneAG} the morphism $ \Phi $ factors as $ q_{1} \circ \Phi_{1} $ where the fibres of $ \Phi_{1} $ are connected and $ q_{1} $ is finite.  If $ \Phi $ is not dominant, then $ \Psi $ is not dominant.  So the image of $ \Psi $ is contained in a closed set.  The image of a connected set is connected, so $ \Psi $ must contract $ D(h) $ to a point.  This is a contradiction so $ \Phi $ is surjective.  
    
    Since $ \Phi \mid_{D(h)} \cong \Psi $, the morphism $ q_{1} \mid_{D(h)} $ is a finite endomorphism of \linebreak $ D(z) \subseteq \mathbb{A}^{1}_{k} $.  Because $ \Psi $ is $ \mathbb{G}_{m} $ equivariant, $ q_{1} \mid_{D(z)} $ is an endomorphism $ q $ of $ \mathbb{G}_{m} $.  The endomorphism $ q $ must be the identity.  This shows that the fibres of $ \Psi $ are connected.  If $ \Psi $ is not generically smooth, let $ L $ be the purely inseparable closure of $ K((\mathbb{G}_{m}^{\ell})) $ in $ \operatorname{Frac}(A) $.  The field $ L $ is an extension of $ k $ of pure transcendence degree one.  So $ L $ is equal to $ k(z) $ for some $ z $ and $ K((\mathbb{G}_{m})^{\ell}) $ is equal to $ k(z^{p^{i}}) $ for some natural number $ i $.  So some power of $ p $ would divide $ q $.  This would mean that $ q $ is not the identity.  Therefore, if $ \Psi $ is a $ \mathbb{G}_{m} $-equivariant morphism which does not factor through a non-trivial endomorphism $ q: \mathbb{G}_{m} \to \mathbb{G}_{m} $ and $ \Psi $ does not contract $ D(h) $ to a point, then the fibres of $ \Psi $ are connected and $ \Psi $ is generically smooth and dominant.
\end{proof}
\begin{prop} \label{P:trivialBundleRed}
    Let $ \operatorname{Spec}(A) $ be an affine $ \mathbb{G}_{m} $-variety over an algebraically closed field $ k $ with action $ \beta $, and let $ D(h) $ be a $ \mathbb{G}_{m} $-equivariant sub-variety.  The following are equivalent:
    \begin{itemize}
        \item[a)] the variety $ D(h) $ is a trivial $ \mathbb{G}_{m} $ bundle over $ D(h)//\mathbb{G}_{m} $,
        \item[b)] there is a $ \mathbb{G}_{m} $-equivariant, generically smooth, dominant, morphism \linebreak $ \Psi: D(h) \to \mathbb{G}_{m} $ with connected fibres,
        \item[c)] there is a rational function $ g/h^{e} \in A_{h}^{\ast} $, which is a semi-invariant of weight one, i.e, $ (g/h^{e})(w \bullet x) = w \left((g/h^{e})(x)\right) $ for all $ w \in \mathbb{G}_{m} $ and $ x \in D(h) $.
    \end{itemize}
\end{prop}
\begin{proof}
    If $ D(h) $ is a trivial $ \mathbb{G}_{m} $-bundle over $ D(h)//\mathbb{G}_{m} $, then there is an isomorphism $ \lambda $ from $ D(h) $ to $ \mathbb{G}_{m} \times D(h)//\mathbb{G}_{m} $.  If $ p_{1}: \mathbb{G}_{m} \times D(h)//\mathbb{G}_{m} \to \mathbb{G}_{m} $ is the natural projection morphism, then $ p_{1} \circ \lambda $ is a $ \mathbb{G}_{m} $-equivariant morphism $ \Psi $.  Because $ A_{h} $ is an integral domain, it has no idempotents.  Therefore $ D(h) $ is connected.  Because $ \lambda $ is an isomorphism, so too is $ Y $.  Therefore the fibres of $ \Psi $ are connected.  Because projections and isomorphisms are generically smooth and dominant, $ \Psi $ is generically smooth and dominant.  So, a) implies b).

    Assume $ \Psi: D(h) \to \mathbb{G}_{m} $ is a generically smooth, dominant $ \mathbb{G}_{m} $-equivariant morphism whose fibres are connected.  Let $ Y $ be the scheme theoretic pre-image of the identity with its reduced induced scheme structure.  Because the fibres of $ \Psi $ are connected, $ Y $ is a variety.  Let $ \gamma $ be the morphism which sends a point $ (w,y) $ of $ \mathbb{G}_{m} \times Y $ to $ w \bullet y $.  We claim that $ \gamma $ is injective.  If $ (w_{1},y_{1}) $ and $ (w_{2},y_{2}) $ are two distinct points of $ \mathbb{G}_{m} \times Y $ such that $ w_{1} \bullet y_{1} = w_{2} \bullet y_{2} $, then $ y_{1} = (w_{1}^{-1}w_{2}) \bullet y_{2} $.  Because $ Y $ is the scheme theoretic pre-image of $ 1 $ under $ \Psi $,
    \begin{align*}
        w_{1}^{-1} w_{2} &= w_{1}^{-1}w_{2} \Psi(y_{2}) \\
        &= \Psi\left((w_{1}^{-1}w_{2})\bullet y_{2} \right) \\
        &= \Psi(y_{1}) \\
        &= 1.
    \end{align*}
    Therefore $ w_{1} $ is equal to $ w_{2} $.  However, this means
    \begin{align*}
        y_{1} &= (w_{1}^{-1} w_{1}) \bullet y_{2} \\
        &= y_{2},
    \end{align*}
    which contradicts our assumption that the points $ (w_{1},y_{1}) $ and $ (w_{2},y_{2}) $ are distinct.  Now let $ \lambda: D(h) \to \mathbb{G}_{m} \times Y $ be the map below:
     \begin{equation*}
        \lambda(x)=(\Psi(x),\Psi(x)^{-1} \bullet x).
     \end{equation*}
     Since $ \lambda $ and $ \gamma $ are inverses to one another, $ \gamma $ is an isomorphism onto its image.  It is also surjective, since $ \lambda $ maps $ D(h) $ to $ \mathbb{G}_{m} \times Y $.  Because $ \Psi $ is generically smooth, so too is $ \lambda $.  Therefore, $ \lambda $ is an isomorphism of $ D(h) $ onto $ \mathbb{G}_{m} \times Y $.  A categorical quotient is unique, so $ Y \cong D(h)//\mathbb{G}_{m} $.  So b) and a) are equivalent.

    If b) holds, then the following diagram commutes:
    \begin{equation} \label{E:52}
    \xymatrix{
        \mathbb{G}_{m} \times D(h) \ar[rr]^{\beta} \ar[d]^{(\operatorname{id}_{\mathbb{G}_{m}},\Psi)} & & D(h)\ar[d]^{\Psi} \\
        \mathbb{G}_{m} \times \mathbb{G}_{m} \ar[rr]^{\mu_{\mathbb{G}_{m}}} & & \mathbb{G}_{m}
        }.
    \end{equation}
    Let $ g \in A $ and $ e \in \mathbb{N}_{0} $ be such that $ \Psi $ maps $ x \in D(h) $ to $ (g/h^{e})(x) $.  If $ x \in D(h) $, and $ \xi \in \mathbb{G}_{m} $, then $ g/h^{e} $ is a semi-invariant of weight one because:
    \begin{align*}
        (g/h^{e})(\xi \bullet x) &= \Psi(\xi \bullet x) \\
        &= \Psi \circ \beta(\xi,x) \\
        &= \mu_{\mathbb{G}_{m}}\circ (\operatorname{id}_{\mathbb{G}_{m}}, \Psi)(\xi,x) \\
        &= \mu_{\mathbb{G}_{m}}\left(\xi, (g/h^{e})(x)\right) \\
        &= \xi (g/h^{e})(x).
    \end{align*}
    Therefore, b) implies c).

    If $ g/h^{e} \in A_{h}^{\ast} $ is a semi-invariant of weight one, then let $ \Psi $ be the morphism which sends a point $ x \in D(h) $ to $ (g/h^{e})(x) $.  Because the following identities hold:
    \begin{align*}
        \Psi \circ \beta(\xi,x) &= \Psi(\xi \bullet x) \\
        &= (g/h^{e})(\xi \bullet x) \\
        &= \xi (g/h^{e})(x) \\
        &= \mu_{\mathbb{G}_{m}}\left( \xi, (g/h^{e})(x)\right) \\
        &= \mu_{\mathbb{G}_{m}} \circ (\operatorname{id}_{\mathbb{G}_{m}}, \Psi)(\xi,x),
    \end{align*}
    the diagram in ~\eqref{E:52} commutes.  As a result, $ \Psi $ is $ \mathbb{G}_{m} $-equivariant.  If $ \Psi $ factored through a non-identity endomorphism $ q: \mathbb{G}_{m} \to \mathbb{G}_{m} $, then there is some $ \mathbb{G}_{m} $-equivariant morphism $ \Psi_{1} $ from $ D(h) $ to $ (\mathbb{G}_{m})^{1/q} $ such that $ q \circ \Psi_{1} = \Psi $.  By Proposition ~\ref{P:generalBundleRed}, $ \Psi_{1} $ would correspond to a semi-invariant of weight $ 1/q $.  This is not possible if $ q>1 $, so by Lemma ~\ref{L:connectedNonIdentityRed} (see page \pageref{L:connectedNonIdentityRed}) the morphism $ \Psi $ is generically smooth, dominant, and the fibres of $ \Psi $ are connected.  So, b) holds if and only if c) holds.
\end{proof}
\bibliographystyle{amsplain}
\bibliography{TheoryOfUnipActions}
\end{document}